\documentclass[11pt, a4paper]{article}
\usepackage[utf8]{inputenc}
\usepackage{bbm, bm}

\usepackage[margin = 1in]{geometry}
\usepackage{amsfonts,amsmath,amssymb,amsthm}
\usepackage{amsthm}
\usepackage{mathtools}
\usepackage{mathrsfs}
\usepackage{graphicx,subcaption}
\usepackage{booktabs}
\usepackage{tabularx}
\usepackage{lipsum}
\usepackage{hyperref}
\usepackage{multicol}
\hypersetup{
   colorlinks = true,
   citecolor = blue,
   urlcolor = teal,
   linkcolor = purple
}

\usepackage{mathrsfs}
\usepackage{multirow}
\usepackage[makeroom]{cancel}
\usepackage{tikz}
\usepackage{framed,color}
\definecolor{shadecolor}{rgb}{1,0.8,0.3}
\usepackage{fancybox}
\usepackage{caption}
\usepackage{algorithm}
\usepackage{algorithmic}
\usepackage{multicol}
\usepackage{aliascnt}
\usepackage{setspace}
\usepackage{longtable}
\allowdisplaybreaks
\usepackage{comment}

\usetikzlibrary{decorations.markings,arrows}

\allowdisplaybreaks

\title{\textbf{Adjusting SPRT for an Asymptotically Efficient Procedure with Finite Number of Applications of Less Effective Treatment}}

\date{}
\vspace{5mm}
\usepackage{authblk}

\author[]{Sampurna Kundu\footnote{Corresponding Author. \href{sampurna.kundu58@gmail.com}{sampurna.kundu58@gmail.com}}}
\author[]{Jayant Jha\footnote{\href{jayantjha@gmail.com}{jayantjha@gmail.com}}}
\author[]{Subir Kumar Bhandari\footnote{\href{subirkumar.bhandari@gmail.com}{subirkumar.bhandari@gmail.com}}}
\affil{\footnotesize Interdisciplinary Statistical Research Unit, Indian Statistical Institute, Kolkata, India}

\usepackage{natbib}
\setcitestyle{authoryear, open={(},close={)}}

\begin{document}

\maketitle
\theoremstyle{definition}
\newtheorem{axiom}{Axiom}
\newtheorem{Remark}{Remark}
\newtheorem{corollary}{Corollary}[section]
\newtheorem{claim}[axiom]{Claim}
\newtheorem{theorem}{Theorem}[section]
\newtheorem{lemma}{Lemma}[section]
\newtheorem{test}{Test Procedure}
\newtheorem{proposition}{Proposition}
\newtheorem{assumption}[theorem]{Assumption}

\newaliascnt{lemmaa}{theorem}
\newtheorem{lemmaa}[lemmaa]{Theorem}
\aliascntresetthe{lemmaa}
\providecommand*{\lemmaautorefname}{Lemma}
\providecommand*{\corollaryautorefname}{Corollary}
\providecommand*{\testautorefname}{Test Procedure}
\providecommand*{\theoremautorefname}{Theorem}
\providecommand*{\propositionautorefname}{Proposition}
\providecommand*{\remarkautorefname}{Remark}



\theoremstyle{plain}
\newtheorem{exa}{Example}
\newtheorem{rem}{Remark}

\theoremstyle{definition}
\newtheorem{definition}{Definition}
\newtheorem{example}{Example}

\begin{abstract}

We propose an adaptive Sequential Probability Ratio Test ($\mathrm{SPRT}$) with a finite expected number of allocations to the less effective treatment. In the classical $\mathrm{SPRT}$ framework, patients are assigned to the two competing treatments one-by-one until a stopping criterion, based on breaching the boundary values which are pre-determined using the Type-I and Type-II error probabilities, is met. This controls the errors at the cost of ethical efficiency as the exposure to the less effective treatment is large. For testing two simple hypotheses, our proposed procedure employs a likelihood ratio--driven adaptive allocation rule, dynamically concentrating sampling effort on the superior population while preserving asymptotic efficiency (in terms of average sample number) comparable to the classical $\mathrm{SPRT}$. We show that at each step, the allocated observations remain independent despite the data-dependent allocation, prove that the expected number of allocations to the inferior treatment remains finite, with an explicit upper bound, and derive a closed-form approximation to it. Extensive simulation studies and data-calibrated case studies substantiate the theoretical results, evincing a significant reduction in inferior allocations compared to the classical $\mathrm{SPRT}$. The proposed design thus offers a balanced method between statistical precision and ethical responsibility, aligning inferential reliability with patient safety.
\end{abstract}




\section{Introduction} \label{P2:Section 1}

Sequential design, first formalized by \cite{Wald1947} through the Sequential Probability Ratio Test ($\mathrm{SPRT}$ henceforth), is a cornerstone of modern statistical decision theory. It departs from fixed-sample schemes by allowing data to be analyzed as they are collected, and the decision to continue or stop sampling depends on accumulated information. The principal goal is to achieve efficient inference — minimizing the expected or average sample number ($\mathrm{ASN}$) for given Type-I and Type-II error probabilities. This framework has since evolved into a variety of methodologies used in industrial quality control (\citealp{Teohetal2022, LiNenes2024}), reliability testing (\citealp{RasayAlinezhad2022, JainandJain1994}), and clinical trials (\citealp{LiJeskeKlein2012, MartensLogan2024}), where resources are limited and early stopping can save both cost and ethical burden. In the classical sequential paradigm, the experimenter observes an equal number of samples from each of two or more populations at each stage and uses likelihood ratios computed from the collected data to update inference until a stopping criterion is satisfied.

However, adaptive sequential designs allocate experimental units (or observations) to competing treatments or populations sequentially, using information accrued so far to guide future allocation. At each stage, the probability of selecting a population depends on its current estimated superiority. Adaptive sequential design embodies the dual objectives — achieving precise inference while safeguarding subjects from excessive exposure to inferior treatments, which has driven a large literature spanning bandit problems, response-adaptive randomization, and sequential testing (see \citealp{BerryFristedt1985, FriedmanFurbergDeMets2010, IvanovaRosenberger2000, Rosenbergeretal2001}). Typical approaches emphasize either optimality of inference (\citealp{WaldWolfitz1948}) or desirable allocation properties (\citealp{HuRosenberger2006}), but rarely provide closed-form control of the total use of the inferior option.

The central problem in adaptive allocation lies in balancing statistical optimality and ethical considerations. Ideally, an adaptive procedure should satisfy the following criteria:
\begin{itemize}
    \item It should asymptotically allocate a higher proportion of samples to the better population and the number of samples allocated to the inferior population should remain small or, if possible, finite in expectation.
    \item It should retain inferential efficiency comparable to the $\mathrm{SPRT}$, ensuring that Type-I and Type-II error probabilities remain controlled.
\end{itemize}

Most existing methods satisfy only part of these goals. Urn and play-the-winner–type (\citealp{Zelen1969, WeiDurham1978}) schemes tend to concentrate sampling on the superior treatment as trials progress, but typically the absolute number of allocations to the inferior treatment continues to grow with the total sample size and becomes infinite eventually. Moreover, very few of these studies provide analytical expressions for the number of inferior allocations or connect their procedures to formal inferential properties such as $\mathrm{ASN}$ or error control.

A number of studies have examined aspects of adaptive allocation and ethical sequential sampling. \citet{BiswasBandyopadhyayBhattacharya2008} provided a comprehensive treatment of response-adaptive designs in clinical trials, highlighting the ethical dimension of minimizing allocations to inferior treatments. Recent contributions by \citet{DasSBhattacharyaBiswas2023}, \citet{BiswasBhattacharyaDasS2020}, \citet{BandyopadhyayMukherjeeBiswas2020}, \citet{DasR2024}, and \citet{DasS2024} have proposed new adaptive rules incorporating covariate adjustments, ordinal responses, crossover trials, multi-treatment response adaptive design, misclassifications, and adaptive interim decisions, all designed to enhance ethical allocation properties.

A related line of work satisfies the second criterion but not the first. The problem considered here belongs to the broader framework of sequential multiple testing with controlled sensing, initiated by \citet{Chernoff1959} and developed in generality for sequential anomaly detection under sampling constraints by \citet{TsopelakosFellouris2023}, in which several data streams, each following $f_0$ or $f_1$, are sampled selectively in order to identify the ``anomalous'' ones with controlled error probabilities and minimal expected stopping time. Here in our paper, although the framework with two densities is somewhat similar, the different goals of the two procedures have to be kept in mind to judge our procedure. The objective of our paper is different and unique in the sense that we always give one of the major priorities to minimizing the number of allocations to the less effective treatment. The problem is also related to best-arm identification with fixed confidence in two-armed bandits \citep{GarivierKaufmann2016}. However, asymptotically optimal policies in these frameworks minimize the average sample number alone and may allocate a positive fraction of all observations to the inferior population, which is unacceptable in clinical applications and also much cruder in the sense of ethical allocation, as can be seen from our results.

Earlier studies explored related versions of the two-treatment adaptive problem, but with important limitations. \citet{Bhandari2007} considered the case with known face value of the parameters, indicating that under a specific adaptive procedure, the expected number of allocations to the less effective treatment could be finite, although no closed-form expression or rigorous distributional analysis was provided, and inferential aspects such as the $\mathrm{ASN}$ or probability of correct selection ($\mathrm{PCS}$ henceforth) were not addressed. \citet{Bhandari2009} extended the problem to unknown parameters and obtained that the expected number of inferior allocations grows logarithmically with the total sample size, but the study did not link the allocation mechanism with inferential efficiency. More recently, \citet{KunduJhaBhandari2025} revisited the problem in an adaptive sequential context, proving that the number of inferior allocations is a finite random variable with finite moments. However, their proof relied on a subset of the sample space corresponding to correct selection events and did not yield a general closed-form expression, while the procedure’s inferential efficiency was only qualitatively assessed but not comparable with the $\mathrm{SPRT}$ framework.

Building on that theoretical foundation, the present paper addresses the above gaps by developing a new adaptive sequential procedure for testing two simple hypotheses, and employs a likelihood ratio–driven adaptive rule that determines at each stage which population to sample next, based on the cumulative log-likelihood ratio of the data collected so far from the population with the larger current sample size. The method ensures that sampling effort is increasingly concentrated on the better-performing population, thereby simultaneously achieving statistical efficiency and ethical prudence. Most importantly, we prove that the expected number of allocations to the less effective treatment remains finite, with an explicit upper bound, and we derive a closed-form approximation to its limit. While the existence of a finite bound was hinted at in our earlier work \citep{KunduJhaBhandari2025}, no closed-form expression had been obtained. In the current framework, using the distributional properties of the standardized partial sums of the log-likelihood ratio statistics, we elucidate that this expected number converges to a finite limit, which is approximated by a closed-form expression depending only on the means and variances of the underlying log-likelihood ratio statistics. This expression provides a quantifiable measure of ethical efficiency. Based on the evaluation of the log-likelihood ratio statistic at each step, we then embed the allocation rule in an adaptive $\mathrm{SPRT}$ and demonstrate that it retains the asymptotic efficiency of the classical $\mathrm{SPRT}$ in terms of $\mathrm{ASN}$, so that ethical efficiency is achieved with little loss in statistical power. Simulation studies and data-calibrated case studies corroborate the theoretical findings and indicate that the empirical results closely match the derived expressions across a range of parameter settings.

The remainder of this paper is organized as follows. Section \ref{P2:Section 2} introduces the formal preliminaries and describes the proposed adaptive allocation rule (Method $\mathcal{M}$). Section \ref{P2:Section 3} explicates the main theoretical results, including an explicit upper bound for the expected number of inferior allocations and a closed-form approximation to its limit, along with the efficiency comparison with the $\mathrm{SPRT}$. Section \ref{P2:Section 4} assesses the results of comprehensive simulation studies under various parameter choices and distributions. The practical utility of the proposed procedure is further illustrated through data-calibrated case studies presented in Section \ref{P2:Section 5}. Finally, Section \ref{P2:Section 6} concludes the paper. Throughout this work, the symbol $\sim$ is used to denote the distribution of a random variable and $\to$ stands for a genuine limit. $\Phi(\cdot)$ denotes the distribution function of the standard normal law. The symbol $\approx$ means that the effective absolute difference between the two sides is negligibly small for all practical purposes.

\section{Preliminaries} \label{P2:Section 2}

Let $\mathcal{X} = X_{1}, X_{2}, X_{3}, \ldots$ and $\mathcal{Y} = Y_{1}, Y_{2}, Y_{3}, \ldots$ two independent data streams be generated. $X_{i}$ and $Y_{i}$ have densities from $\{f_{0}, f_{1}\}$ with respect to some $\sigma$-finite measure. It is not known which $f_{j} \,(f_{0} \text{ or } f_{1})$ is assigned with $X_{i}$ or $Y_{i}$.

\subsection{Adaptive Sequential Procedure} \label{P2:Section 2.1}

We start with one sample from $\mathcal{X}$ or $\mathcal{Y}$ chosen at random. At step $n$, using past data, we use Method $\mathcal{M}$ to select the population from which to collect a sample next. At step $n$, let we have $N_{0, n}$ and $N_{1, n}$ samples from $\mathcal{X}$ and $\mathcal{Y}$ respectively, with $N_{0, n} + N_{1, n} = n$.
Let us define the following:
$$n_{max} = \max \{N_{0, n}, N_{1, n}\}\quad \text{and} \quad n_{min} = \min \{N_{0, n}, N_{1, n}\}.$$

When $N_{0, n} = N_{1, n}$, $n_{max}$ is chosen as either $N_{0, n}$ or $N_{1, n}$ with probability $\frac{1}{2}$ each, and $n_{min}$ is defined to be the remaining one.

Consider two simple null and alternative hypotheses:
$H_{0}: (f_{0}, f_{1})$ vs $H_{1}: (f_{1}, f_{0})$, where the first coordinate represents density that corresponds to $\mathcal{X}$.

\begin{theorem} \label{P2:Theorem 2.1}
At step $n$, the samples collected from $\mathcal{X}$ and
$\mathcal{Y}$ are respectively
\[
(X_1,X_2,\ldots,X_{N_{0,n}})
\quad\text{and}\quad
(Y_1,Y_2,\ldots,Y_{N_{1,n}}), \quad \text{where } N_{0,n}+N_{1,n}=n.
\]
The $n$ samples together are independent, the distribution of
$X_{1}, X_{2}, \ldots, X_{N_{0,n}}$ is i.i.d., and that of
$Y_{1}, Y_{2}, \ldots, Y_{N_{1,n}}$ is also i.i.d.
\end{theorem}

\begin{proof}
    Let $T_k \in \{0,1\}$ indicate whether the $k$-th sample is drawn from $\mathcal{X}$ ($T_k=0$) or $\mathcal{Y}$ ($T_k=1$). So, as a result we get a sequence of $0$'s and $1$'s (with $N_{0,n}$ many $0$'s and $N_{1,n}$ many $1$'s). This sequence is constructed with the $0$'s and $1$'s according to the adaptive allocation rule. Effectively, we get a sequence of $X_i$'s and $Y_j$'s with $i=1, \ldots, N_{0,n}$ and $j=1, \ldots, N_{1,n}$ where $X_i$ replaces $i$-th $0$ and $Y_j$ replaces $j$-th $1$ in the above $0$--$1$ sequence, respectively.

    So, effectively we get the data as the above sequence of $X_i$'s and $Y_j$'s. The elements of this sequence are values of purely independent observations coming from $\mathcal{X}$-stream or $\mathcal{Y}$-stream. Here, $f_X$ and $f_Y$ denote the densities of the $\mathcal{X}$ and $\mathcal{Y}$-streams, i.e., $(f_X, f_Y) = (f_0, f_1)$ under $H_0$ and $(f_1, f_0)$ under $H_1$, respectively.
    
    Hence, considering the pdf or pmf of the sequence vector, we get
    \begin{align*}
        f(\text{ordered sequence of }X_i\text{'s and }Y_j\text{'s with a total of }n\text{ many})
        &= \prod_{i=1}^{N_{0,n}} f_X(X_i)\prod_{j=1}^{N_{1,n}} f_Y(Y_j).
    \end{align*}
\end{proof}

\begin{Remark} \label{P2:Remark 1}
   We consider, without loss of generality, $f_0$ to be the better distribution. Moreover, we consider null hypothesis $H_0$: $X \sim f_0$ and alternative hypothesis $H_1$: $Y \sim f_0$. Without loss of generality, we consider that $H_0$ is true.
\end{Remark}

\begin{Remark}\label{P2:Remark 2}
Here, both hypotheses are simple. If they were characterized by a parameter $\theta$, the better population would conventionally be the one with the higher (or, lower) value of $\theta$. But in effect, assigning one of the two densities in the simple hypotheses as the better one (here $f_0$) therefore imposes no restriction on the general set-up. In the process of usual inference, the procedure treats the two streams symmetrically: if $\mathcal{X}$ follows the better density, it is selected with high probability, and likewise for $\mathcal{Y}$.
\end{Remark}

\subsection{Method \texorpdfstring{$\mathcal{M}$}{M}} \label{P2:Section 2.2}

One observes that $n_{max} \geq \frac{n}{2}$. At step $n$, we consider $U_{1}, U_{2}, \ldots, U_{n_{max}}$ the i.i.d. sample where $U_i$ ($X_i$ or $Y_i$) corresponds to $n_{max}$.

\begin{enumerate}
    \item[(i)] If $\log{\left[\frac{\prod_{i=1}^{n_{max}} f_0{(U_i)}}{\prod_{i=1}^{n_{max}} f_1{(U_i)}}\right]}$ $>$ $0$, we draw one more sample from $f_j$ corresponding to $n_{max}$.
    
    \item[(ii)] If $\log{\left[\frac{\prod_{i=1}^{n_{max}} f_0{(U_i)}}{\prod_{i=1}^{n_{max}} f_1{(U_i)}}\right]}$ $<$ $0$, we draw one more sample from $f_j$ corresponding to $n_{min}$.
    
    \item[(iii)] If $\log{\left[\frac{\prod_{i=1}^{n_{max}} f_0{(U_i)}}{\prod_{i=1}^{n_{max}} f_1{(U_i)}}\right]}$ $=$ $0$, we draw a sample from $f_0$ or $f_1$ with probability $\frac{1}{2}$ each.
\end{enumerate}

\section{Main Result} \label{P2:Section 3}

We aim to allocate more samples to the better density $f_0$. Let $\mathbb{PI}_n$ denote the probability of incorrect allocation at step $n$. Then
\begin{align*}
    \mathbb{PI}_n
    &= \mathbb{P}(\text{not allocating to the population corresponding to } n_{max} \text{ at step } n,\\
    &\qquad \text{ when } \log{\left[\frac{\prod_{i=1}^{n_{max}} f_0{(U_i)}}{\prod_{i=1}^{n_{max}} f_1{(U_i)}}\right]} > 0)\\
    &\quad + \mathbb{P}(\text{not allocating to the population corresponding to } n_{min} \text{ at step } n,\\
    &\qquad \text{ when } \log{\left[\frac{\prod_{i=1}^{n_{max}} f_0{(U_i)}}{\prod_{i=1}^{n_{max}} f_1{(U_i)}}\right]} < 0).
\end{align*}

Of course, $n_{\max}$ and $n_{min}$ are random objects, which are functions of $n$ and the data. This in no way contradicts, mathematically or otherwise, the definition of $\mathbb{PI}_n$.

Throughout the theoretical development, we consider $f_0$ and $f_1$ to be continuous densities. For discontinuous distributions, the determination of the limiting expected number of allocations to the inferior treatment turns out to be considerably more involved, as there does not appear to be a uniform methodology analogous to that available for continuous densities. In fact, a complete account of such cases would require several distribution-specific technical considerations, leading to the examination of many isolated cases. Consequently, we restrict our theoretical development to continuous densities. To illustrate the behaviour of the proposed procedure in a discontinuous setting, the Poisson distribution has nevertheless been incorporated into the simulation study (Section \ref{P2:Section 4}), where the limiting expected number of allocations to the less effective treatment exhibits a moderate departure from $N_1^*$ (the approximation derived in Theorem \ref{P2:Theorem 3.2}, for continuous densities). In the following, the expression is derived for continuous densities.

\noindent\textit{Notation.} Write $n_x = N_{0,n}$ and $n_y = N_{1,n}$, and define
\[
  Z^{(X)}_i = \log\frac{f_0(X_i)}{f_1(X_i)}, \qquad
  Z^{(Y)}_j = \log\frac{f_0(Y_j)}{f_1(Y_j)},
\]
with means $\eta_x = \mathbb{E}\big(Z^{(X)}_i\big) > 0$ and
$\eta_y = \mathbb{E}\big(Z^{(Y)}_j\big) < 0$, (under the assumption that $H_0$ is true, i.e., $X \sim f_0$, which has been mentioned in Remark \ref{P2:Remark 1} without loss of generality) and also let the respective variances of the variables $Z^{(X)}_i$ and $Z^{(Y)}_j$ be $\sigma_x^2$ and $\sigma_y^2$. By Theorem \ref{P2:Theorem 2.1}, $\{Z^{(X)}_i\}$ and $\{Z^{(Y)}_j\}$ are independent i.i.d.\ sequences. We also define the centred variables
\[
  W^{(X)}_i = \eta_x - Z^{(X)}_i, \qquad W^{(Y)}_j = Z^{(Y)}_j - \eta_y,
\]
which have mean $0$ and variances $\sigma_x^2$ and $\sigma_y^2$, respectively. Below, ``$U=\mathcal{X}$'' (respectively, ``$U=\mathcal{Y}$'') means that the $n_{\max}$-stream
$U_1, U_2, \ldots$ is the $\mathcal{X}$-stream (respectively the $\mathcal{Y}$-stream). Then
\begin{align}
  \mathbb{PI}_n
  &=
  \begin{cases}
    \mathbb{P}\left[\log\left(\prod_{i=1}^{n_{\max}} \frac{f_0(U_i)}{f_1(U_i)}\right) < 0\right],
      & \text{if } U = \mathcal{X},\\[8pt]
    \mathbb{P}\left[\log\left(\prod_{j=1}^{n_{\max}} \frac{f_0(U_j)}{f_1(U_j)}\right) > 0\right],
      & \text{if } U = \mathcal{Y},
  \end{cases}
  \nonumber\\[4pt]
  &=
  \begin{cases}
    \mathbb{P}\left[\sum_{i=1}^{n_x} Z^{(X)}_i < 0\right], & \text{if } U = \mathcal{X},\\[6pt]
    \mathbb{P}\left[\sum_{j=1}^{n_y} Z^{(Y)}_j > 0\right], & \text{if } U = \mathcal{Y},
  \end{cases}
  \nonumber\\[4pt]
  &=
  \begin{cases}
    \mathbb{P}\left[\sum_{i=1}^{n_x} W^{(X)}_i > n_x \eta_x\right], & \text{if } U = \mathcal{X},\\[6pt]
    \mathbb{P}\left[\sum_{j=1}^{n_y} W^{(Y)}_j > -n_y \eta_y\right], & \text{if } U = \mathcal{Y},
  \end{cases}
  \label{P2:Equation 1}
  \end{align}
Hence,
\begin{equation}
\begin{split}
  \sum_{m=1}^{n} \mathbb{PI}_m
  \;\leq\; & \sum_{m=1}^{\infty} \mathbb{P}\left[\frac{1}{\sigma_x\sqrt{m}} \sum_{i=1}^{m} W^{(X)}_i > \frac{\eta_x}{\sigma_x}\sqrt{m}\,\right] \\
  & + \sum_{m=1}^{\infty} \mathbb{P}\left[\frac{1}{\sigma_y\sqrt{m}} \sum_{j=1}^{m} W^{(Y)}_j > -\frac{\eta_y}{\sigma_y}\sqrt{m}\,\right].
\end{split}
\label{P2:Equation 2}
\end{equation}
and, the expected number of allocations to the less effective treatment is
\begin{equation} \label{P2:Equation 3}
  \mathbb{E}(N_{1,n}) = \sum_{m=1}^{n} \mathbb{PI}_m .
\end{equation}
 
\subsection{An upper bound for the expected number of inferior allocations} \label{P2:Section 3.1}

\begin{assumption} \label{P2:Assumption 3.1}
The fourth moments $\mu_4^{(X)} = \mathbb{E}\big[(W^{(X)}_i)^4\big]$ and
$\mu_4^{(Y)} = \mathbb{E}\big[(W^{(Y)}_j)^4\big]$ are finite.
\end{assumption}

\begin{lemma}[Upper bound for the expected number of inferior allocations] \label{P2:Lemma 3.1}
Under Assumption \ref{P2:Assumption 3.1}, for every $n$,
\begin{align*}
  \mathbb{E}(N_{1,n}) \le N_2^* :=\;&
  \frac{\pi^2}{2}\left[\frac{\sigma_x^4}{\eta_x^4} + \frac{\sigma_y^4}{\eta_y^4}\right] + \zeta(3)\left[\frac{\big(\mu_4^{(X)}-3\sigma_x^4\big)^+}{\eta_x^4}
  + \frac{\big(\mu_4^{(Y)}-3\sigma_y^4\big)^+}{\eta_y^4}\right],
\end{align*}
where $\zeta(3)=\sum_{i=1}^{\infty} i^{-3}\approx 1.202$ and $a^+=\max\{a,0\}$. Consequently, $\lim\limits_{n\to\infty}\mathbb{E}(N_{1,n})$ exists and is at most $N_2^* < \infty$.
\end{lemma}

\begin{proof}
Since the $W^{(X)}_i$ are i.i.d.\ with mean zero, all cross terms in the fourth power with an unpaired index vanish, and
\begin{align*}
  \mathbb{E}\left[\left(\sum_{i=1}^{m} W^{(X)}_i\right)^4\right]
  &= \sum_{i} \mathbb{E}\big[(W^{(X)}_i)^4\big]
    + 6\sum_{i<i'} \mathbb{E}\big[(W^{(X)}_i)^2\big]\,\mathbb{E}\big[(W^{(X)}_{i'})^2\big]\\
  &= m\mu_4^{(X)} + 3m(m-1)\sigma_x^4\\
  &= 3m^2\sigma_x^4 + m\big(\mu_4^{(X)}-3\sigma_x^4\big)
   \;\le\; 3m^2\sigma_x^4 + m\big(\mu_4^{(X)}-3\sigma_x^4\big)^+ .
\end{align*}
Similarly,
$\mathbb{E}\big[\big(\sum_{j=1}^{m} W^{(Y)}_j\big)^4\big]
\le 3m^2\sigma_y^4 + m\big(\mu_4^{(Y)}-3\sigma_y^4\big)^+$.
Since $m\eta_x > 0$ and $-m\eta_y > 0$, applying Markov's inequality to \eqref{P2:Equation 1} and from \eqref{P2:Equation 3} gives
\begin{align*}
  \mathbb{E}(N_{1,n})
  &= \sum_{m=1}^{n} \mathbb{PI}_m \\
  &\leq \sum_{m=1}^{\infty} \frac{\mathbb{E}\Big[\big(\sum_{i=1}^{m} W^{(X)}_i\big)^4\Big]}{(m\eta_x)^4}
      + \sum_{m=1}^{\infty} \frac{\mathbb{E}\Big[\big(\sum_{j=1}^{m} W^{(Y)}_j\big)^4\Big]}{(m\eta_y)^4} \\
  &\leq \sum_{m=1}^{\infty} \left[\frac{3\sigma_x^4}{\eta_x^4}\,\frac{1}{m^2}
      + \frac{\big(\mu_4^{(X)} - 3\sigma_x^4\big)^+}{\eta_x^4}\,\frac{1}{m^3}\right] + \sum_{m=1}^{\infty} \left[\frac{3\sigma_y^4}{\eta_y^4}\,\frac{1}{m^2}
      + \frac{\big(\mu_4^{(Y)} - 3\sigma_y^4\big)^+}{\eta_y^4}\,\frac{1}{m^3}\right] \\
  &= \frac{\pi^2}{2}\left(\frac{\sigma_x^4}{\eta_x^4} + \frac{\sigma_y^4}{\eta_y^4}\right)
      + \zeta(3)\left[\frac{\big(\mu_4^{(X)} - 3\sigma_x^4\big)^+}{\eta_x^4}
      + \frac{\big(\mu_4^{(Y)} - 3\sigma_y^4\big)^+}{\eta_y^4}\right] \\
  &= N_2^* < \infty,
\end{align*}
since $\sum_{m=1}^{\infty} m^{-2} = \pi^2/6$ and $\sum_{m=1}^{\infty} m^{-3} = \zeta(3)$.
Since $N_{1,n}$ is non-decreasing in $n$, $\mathbb{E}(N_{1,n})$ is non-decreasing and bounded by $N_2^*$. Hence its limit exists and does not exceed $N_2^*$.
\end{proof}

\subsection{Normal approximation and the closed-form approximation \texorpdfstring{$N_1^*$}{N1*}} \label{P2:Section 3.2}

\begin{Remark}[Normal approximation of the terms of $\mathbb{E}(N_{1,n})$] \label{P2:Remark 3}
Each term of $\mathbb{E}(N_{1,n})$ in \eqref{P2:Equation 3} is the value of the distribution function of a standardized sum of i.i.d.\ log-likelihood ratios, at a point moving to $\mp\infty$. Under mild moment conditions, replacing these terms by the distribution function of a normal distribution changes $\mathbb{E}(N_{1,n})$ only by a small bounded amount; this follows from the Edgeworth expansion.
\end{Remark}

\subsubsection{Justification via the Edgeworth expansion} \label{P2:Section 3.2.1}

Write $c_x = \eta_x/\sigma_x > 0$ and $c_y = -\eta_y/\sigma_y > 0$. Let $F_i$ and $G_j$ be the
distribution functions of the standardized sums
$\sum_{l=1}^{i}(Z^{(X)}_l-\eta_x)/(\sigma_x\sqrt{i})$ and
$\sum_{l=1}^{j}(Z^{(Y)}_l-\eta_y)/(\sigma_y\sqrt{j})$. By \eqref{P2:Equation 2}, the terms of
$\mathbb{E}(N_{1,n})$ are
\[
  \mathbb{P}\left[\sum_{l=1}^{i} Z^{(X)}_l \le 0\right] = F_i\left(-c_x\sqrt{i}\right),
  \qquad
  \mathbb{P}\left[\sum_{l=1}^{j} Z^{(Y)}_l > 0\right] = 1 - G_j\left(c_y\sqrt{j}\right).
\]
The central limit theorem states only that $F_i(x)\to\Phi(x)$ for each fixed $x$. Here, however, $F_i$ is evaluated at the moving point $x_i = -c_x\sqrt{i}\to-\infty$. The Edgeworth expansion measures how close the two are.

Suppose that the characteristic function $\psi(t)=\mathbb{E}\big(e^{\mathrm{i}tZ^{(X)}_1}\big)$ of
$Z^{(X)}_1$ satisfies
\[
\limsup_{|t|\to\infty}|\psi(t)|<1,
\]
(i.e., the distribution satisfies the strong non-lattice condition), and that the moments $\mu_3,\ldots,\mu_r$ are finite. In the present setting, when the underlying observations have continuous distributions and hence, the corresponding log-likelihood ratio variables also possess absolutely continuous distributions, the above non-lattice condition is automatically satisfied. As can be seen from our subsequent derivations, the different cases involving continuous distributions admit almost a similar treatment, and we present a common analytical approach for these cases below. Thus, in this context, only the assumption of finite moments is needed.

As explained at the beginning of Section \ref{P2:Section 3}, discontinuous distributions do not admit a uniformly treatable analytical approach and are therefore excluded from the theoretical development; however, the Poisson case is considered in the simulation study (Section \ref{P2:Section 4}). The approximation discrepancy in $N_1^*$ for the Poisson distribution may be less accurate as it is not in the continuous density context but of course the upper bound of $\lim\limits_{n \to \infty}\mathbb{E}(N_{1,n})$ effectively works with $N_2^*$ (though very crude).

Now, as $i \to \infty$,
\begin{equation} \label{P2:Equation 4}
  F_i(x) - \Phi(x) - \phi(x)\sum_{k=3}^{r} i^{-k/2+1}R_k(x) = o\big(i^{-r/2+1}\big), \quad \text{ uniformly in } x,
\end{equation}
where $\phi(\cdot)$ is the standard normal density and each $R_k$ is a polynomial that depends only on $\mu_1,\ldots,\mu_r$ and not on $i$ \citep[see][Chapter~XVI, Section~4, Theorem~3]{Feller1971}.

For $r=3$, \eqref{P2:Equation 4} is the familiar first-order
correction for skewness:
\[
  F_i(x) = \Phi(x) - \frac{\gamma_x}{6\sqrt{i}}\left(x^2-1\right)\phi(x) + o\big(i^{-1/2}\big),
  \qquad \gamma_x = \frac{\mu_3}{\sigma_x^3}.
\]
At $x_i = -c_x\sqrt{i}$, each correction term has the form
$\phi(x_i)\times(\text{a polynomial in }\sqrt{i})$, and $\phi(x_i)\propto e^{-c_x^2 i/2}$. Hence, every correction term is exponentially small. For instance, the skewness term equals
\[
  \frac{|\gamma_x|}{6\sqrt{i}}\left(c_x^2 i-1\right)\phi\left(c_x\sqrt{i}\right)
  = O\left(\sqrt{i}\,e^{-c_x^2 i/2}\right).
\]
Since an exponentially small quantity is smaller than any power of $i$, these terms are $o\big(i^{-r/2+1}\big)$, the same order as the remainder in \eqref{P2:Equation 4}. It follows that
\[
  \left|F_i\left(-c_x\sqrt{i}\right) - \Phi\left(-c_x\sqrt{i}\right)\right| = o\big(i^{-r/2+1}\big),
\]
and similarly the same expression holds for $i$ replaced by $j$, $c_x$ replaced by $c_y$, and $F_i$ replaced by $G_j$. If the log-likelihood ratio has finite moments of all orders, this discrepancy therefore decays faster than any power of $i$. In the process of calculating $\mathbb{E}(N_{1,n})$, the respective discrepancies are then summable over $i$ and the sum, i.e., the cumulated error is dominated mostly by the first few terms. Consequently, replacing each probability in \eqref{P2:Equation 3} by the values with the CDF of normal variate, changes $\mathbb{E}(N_{1,n})$ only by a small, bounded amount. This is the justification for the CLT step below.

Hence, from \eqref{P2:Equation 2} and \eqref{P2:Equation 3}, for all $n$,
\begin{align}
  \mathbb{E}(N_{1,n})
  &\le \sum_{i=1}^{\infty}\Phi\left(-\frac{\eta_x}{\sigma_x}\sqrt{i}\right)
        + \sum_{j=1}^{\infty}\Phi\left(\frac{\eta_y}{\sigma_y}\sqrt{j}\right) < \infty . \label{P2:Equation 5}
\end{align}

The bound $\mathbb{E}(N_{1,n}) \le N_2^*$ of Lemma \ref{P2:Lemma 3.1}, in contrast, holds without any normal approximation.

Using the normal approximation justified above, a closed-form approximation to $\lim\limits_{n \to \infty} E(N_{1,n})$ is given by the following result.
 
\begin{theorem}[Closed-form approximation of the limiting expected number of inferior allocations] \label{P2:Theorem 3.2}
The expected number of allocations to the less effective treatment under the proposed adaptive rule (and under the assumptions of Sections \ref{P2:Section 2} and \ref{P2:Section 3}) satisfies
\[
  \lim_{n \to \infty}\mathbb{E}(N_{1,n}) \approx N_1^*:= \frac{1}{2}\left(\frac{\sigma_x^2}{\eta_x^2} + \frac{\sigma_y^2}{\eta_y^2}\right).
\]
In this way $N_1^*$ is defined as a finite constant that depends only on the first two moments of the log-likelihood ratio statistics.
\end{theorem}
 
\begin{proof}
Starting from \eqref{P2:Equation 5} and applying the normal approximation to the tail probabilities (Remark \ref{P2:Remark 3}),
\begin{align*}
  \lim_{n \to \infty}\mathbb{E}(N_{1,n})
  &\approx \sum_{i=1}^{\infty}\Phi\left(-\frac{\eta_x}{\sigma_x}\sqrt{i}\right)
         + \sum_{j=1}^{\infty}\Phi\left(\frac{\eta_y}{\sigma_y}\sqrt{j}\right)\\
  &\approx \int_{-\infty}^{0}\Phi\left(-\frac{\eta_x}{\sigma_x}\sqrt{-t}\right)dt
         + \int_{-\infty}^{0}\Phi\left(\frac{\eta_y}{\sigma_y}\sqrt{-t}\right)dt\\
  &= \frac{1}{2}\left(\frac{\sigma_x^2}{\eta_x^2} + \frac{\sigma_y^2}{\eta_y^2}\right) \qquad [\text{using integration by parts}] \\
  &= N_1^*.
\end{align*}
Indeed, for $c>0$, substituting $s=-t$ and integrating by parts,
\begin{align*}
  \int_{-\infty}^{0}\Phi\left(-c\sqrt{-t}\right)dt
  &= \Big[s\,\Phi(-c\sqrt{s})\Big]_0^\infty + \frac{c}{2}\int_0^\infty \sqrt{s}\,\phi(c\sqrt{s})\,ds\\
  &= \frac{1}{c^2}\int_0^\infty u^2\phi(u)\,du = \frac{1}{2c^2},
\end{align*}
and we take $c=\eta_x/\sigma_x$ and $c=-\eta_y/\sigma_y$, in the above two integral contexts for $\lim\limits_{n \to \infty}\mathbb{E}(N_{1,n})$. This gives the approximate value of the number of applications of the less effective treatment for large $n$.
\end{proof}

\begin{Remark} \label{P2:Remark 4}
    Note that the result for the Edgeworth expansion holds for large values of $i$. However, the contribution for small values of $i$ is small for practically all purposes. We verify this discrepancy in the simulation section for some distributions.
\end{Remark}

\begin{Remark} \label{P2:Remark 5}
    $\mathbb{E}(N_{1,n})$ increases to a finite quantity, as $n \to \infty$. Hence, by the monotone convergence theorem,
    $$\mathbb{E}(\lim_{n \to \infty} N_{1,n}) = \lim_{n \to \infty} \mathbb{E}(N_{1,n}) < \infty.$$
    This implies that, $\lim\limits_{n \to \infty} N_{1,n} < \infty$ with probability $1$.
\end{Remark}

\subsection{Stopping Rule} \label{P2:Section 3.3}





So far, we have studied the allocation rule (Method $\mathcal{M}$) alone, which determines from which population each observation is drawn; the results of Sections \ref{P2:Section 3.1} -- \ref{P2:Section 3.2} depend only on this rule. We now introduce the stopping rule, which determines when sampling terminates and which population is selected as the better one. Together, these two components constitute the proposed adaptive $\mathrm{SPRT}$ rule.

We will perform the $\mathrm{SPRT}$ with $U_{1}, U_{2}, \ldots$ to test the hypotheses $K_0 : U_i \sim f_0$ vs $K_1 : U_i \sim f_1$, where $U_i$ corresponds to the data stream for $n_{max}$.

\textit{Final Decision.} If $K_0$ is accepted, we accept $H_0$ as our decision, and if $K_0$ is rejected, we accept $H_1$ as our decision. Note that, in the context of this testing problem, we effectively mean the critical region to be $\{K_0 \text{ is rejected}\}$.

Here, we declare that a Type-I error occurs if $K_0$ is rejected under the null hypothesis $H_0$, and a Type-II error occurs if $K_1$ is rejected under the alternative hypothesis $H_1$ (whether $n_{\max}$ corresponds to the $\mathcal{X}$-data stream or to the $\mathcal{Y}$-data stream is, of course, not known, except with some probability, high or low). Below, we derive the expression for the probability of incorrect selection ($\mathrm{PICS}$).

\begin{Remark} \label{P2:Remark 6}
    In summary, at each step $n$, we observe $U_{1}, U_{2}, \ldots, U_{n_{max}}$, apply Method $\mathcal{M}$ to determine the population for the next treatment allocation, and update $U_{1}, U_{2}, U_{3}, \ldots$ accordingly. The procedure thus generates the $U$-data stream adaptively and terminates when the $\mathrm{SPRT}$ between $K_0$ and $K_1$ stops.
\end{Remark}

\subsection{Details of Adaptive \texorpdfstring{$\mathrm{SPRT}$}{SPRT} Rule} \label{P2:Section 3.4}

Suppose $\alpha = \mathbb{P}(\text{Type-I error}) = \mathrm{PICS}_I$ and $\beta = \mathbb{P}(\text{Type-II error}) = \mathrm{PICS}_{II}$. Let $a \approx \log\!\left(\frac{1-\beta}{\alpha}\right)$ and $b \approx \log\!\left(\frac{\beta}{1-\alpha}\right)$. We consider the following adaptive $\mathrm{SPRT}$ rule:
\begin{itemize}
    \item We continue adaptive sampling if
$b < \sum_{i=1}^{n_{\max}}\log\!\left(\frac{f_1(U_i)}{f_0(U_i)}\right) < a$.
    \item We stop adaptive sampling in favour of $K_1$ if
$\sum_{i=1}^{n_{\max}}\log\!\left(\frac{f_1(U_i)}{f_0(U_i)}\right) \geq a$.
    \item We stop adaptive sampling in favour of $K_0$ if
$\sum_{i=1}^{n_{\max}}\log\!\left(\frac{f_1(U_i)}{f_0(U_i)}\right) \leq b$.
\end{itemize}

Let $\mathrm{ASN}_{A}$ denote the average sample number of the proposed adaptive $\mathrm{SPRT}$ rule.
\begin{Remark} \label{P2:Remark 7}
    If we perform the usual $\mathrm{SPRT}$ with the $\mathcal{X}$-data stream only (or the $\mathcal{Y}$-data stream only) instead of the $U$-data stream, for the hypotheses $K_0: X \sim f_0$ vs $K_1: X \sim f_1$ (or $K_0: Y \sim f_0$ vs $K_1: Y \sim f_1$), let $\mathrm{ASN}_{K_0}$ and $\mathrm{ASN}_{K_1}$ denote the average sample numbers under $K_0$ and $K_1$, respectively. From \cite{Rao1973} (p. $479$), we get asymptotic approximate expressions for the average sample number of the $\mathrm{SPRT}$ as
\begin{equation} \label{P2:Equation 6}
    \mathrm{ASN}_{K_0} \approx \frac{b(1-\alpha) + a\alpha}{-\eta_x} \quad \text{ and, } \quad \mathrm{ASN}_{K_1} \approx \frac{b\beta + a(1-\beta)}{-\eta_y},
\end{equation}
$\alpha$ and $\beta$ are small (tend to $0$) and accordingly $a \to \infty$ and $b \to -\infty$.
Hence, from \eqref{P2:Equation 6},
\begin{equation} \label{P2:Equation 7}
    \begin{aligned}
        \log(\mathrm{PICS}_{II}) \approx -\eta_x \cdot \mathrm{ASN}_{K_0} + o(\mathrm{ASN}_{K_0}), \\
        \log(\mathrm{PICS}_{I}) \approx \eta_y \cdot \mathrm{ASN}_{K_1} + o(\mathrm{ASN}_{K_1}).
    \end{aligned}
\end{equation}
Thus, we get the asymptotic expression for $\log(\mathrm{PICS})$ for the usual $\mathrm{SPRT}$.

We get similar asymptotic expressions for $\log(\mathrm{PICS}_{II})$ and $\log(\mathrm{PICS}_{I})$ in the adaptive $\mathrm{SPRT}$ context. The only difference is in the average sample number. In the adaptive $\mathrm{SPRT}$ case, the corresponding average sample number ($\mathrm{ASN}_A$) can be given by the following asymptotic approximations,
\begin{equation} \label{P2:Equation 8}
        \mathrm{ASN}_A \approx \mathrm{ASN}_{K_0} + N_1^\ast \text{ (or, } \approx \mathrm{ASN}_{K_1} + N_1^\ast)
    \end{equation}
    where $N_1^*$ is given in Theorem \ref{P2:Theorem 3.2}.
\end{Remark}

As the additional term $N_1^*$ in \eqref{P2:Equation 8} remains finite while the average sample number grows, we obtain the following result.

\begin{theorem}[Asymptotic efficiency of the proposed selection procedure] \label{P2:Theorem 3.3}
    Under the assumptions in Section \ref{P2:Section 2}, the proposed adaptive $\mathrm{SPRT}$ rule using $U_{1}, U_{2}, \ldots, U_{n_{max}}$ is asymptotically efficient, and for the same value of $\mathrm{PICS}$,
    \begin{align*}
        \frac{\mathrm{ASN}_A}{\mathrm{ASN}_{K_0}} \to 1 \quad \text{ or, } \quad
        \frac{\mathrm{ASN}_A}{\mathrm{ASN}_{K_1}} \to 1. 
    \end{align*}
\end{theorem}

\begin{proof}
    The proof can be directly followed from the \eqref{P2:Equation 8} for the adaptive $\mathrm{SPRT}$ rule.
\end{proof}

\begin{Remark}[Probability of correct selection] \label{P2:Remark 8}
Corresponding to the two kinds of incorrect selection in \eqref{P2:Equation 7}, we define the probabilities of
correct selection
\[
  \mathrm{PCS}_{I} = 1 - \mathrm{PICS}_{I},
  \qquad
  \mathrm{PCS}_{II} = 1 - \mathrm{PICS}_{II}.
\]
Thus, $\mathrm{PCS}_{I}$ and $\mathrm{PCS}_{II}$ are the probabilities that the $\mathrm{SPRT}$ accepts $K_0$ when the $U$-stream follows $f_0$ and accepts $K_1$ when it follows $f_1$, respectively, so that in either case the better population is correctly selected. Using $N_1^*$, the number of inferior allocations as in \eqref{P2:Equation 8}, we get the following approximate equalities for the proposed adaptive rule,
\[
  \mathrm{PCS}_{II} \approx 1 - e^{-\eta_x(\mathrm{ASN}_A - N_1^*) + o(\mathrm{ASN}_A)},
  \qquad
  \mathrm{PCS}_{I} \approx 1 - e^{\eta_y(\mathrm{ASN}_A - N_1^*) + o(\mathrm{ASN}_A)},
\]
where $\mathrm{ASN}_A \approx \mathrm{ASN}_{K_0} + N_1^*$ (respectively, $\mathrm{ASN}_{K_1} + N_1^*$) and $\eta_x > 0$, $\eta_y < 0$. Both tend to $1$ as average sample number increases. Compared with the classical $\mathrm{SPRT}$, the loss in $\mathrm{PCS}$ in context of adaptive $\mathrm{SPRT}$ is incorporated only through the finite amount of increase by $N_1^*$ in the exponent (Remark \ref{P2:Remark 15}). The above discussion relates probability of correct selection in the context of adaptive $\mathrm{SPRT}$ and usual $\mathrm{SPRT}$ when both allow the same value of average sample number.
\end{Remark}

\begin{Remark} \label{P2:Remark 9}
    As described in Remark \ref{P2:Remark 5}, $\lim\limits_{n \to \infty} N_{1,n}$ is a finite random variable. Hence, in the sequence $U_{1}, U_{2}, U_{3}, \ldots$ that we obtain, all but finitely many random variables effectively come either from the $\mathcal{X}$-stream or from the $\mathcal{Y}$-stream. All the results that we have proved and considered in the context of the adaptive $\mathrm{SPRT}$, or of approximating it by the usual $\mathrm{SPRT}$, are asymptotic. As such, all the proofs in this connection, whether in the context of the probability of correct selection, the average sample number or any other quantity, are validated in this asymptotic sense.
    
    In all the simulation tables, it is found that $N_1^*$ is small or moderately small compared with the average sample number. Hence, all the results that we have obtained are adequately accurate as approximations.
\end{Remark}

\subsection{Example} \label{P2:Section 3.5}

We illustrate the results with an example of $f_0 \sim \operatorname{N}(\theta_{0}, 1)$ (first data stream) and $f_1 \sim \operatorname{N}(\theta_{1}, 1)$ (second data stream). Then with the notations of Section \ref{P2:Section 3} we have, $\eta_x = \frac{1}{2}  (\theta_0 - \theta_1)^2 > 0$ and $\eta_y = -\frac{1}{2}  (\theta_0 - \theta_1)^2 < 0$ are the means of $Z_i^{(X)} = \log\!\left(\tfrac{f_0(X_i)}{f_1(X_i)}\right)$ and $Z_i^{(Y)} = \log\!\left(\tfrac{f_0(Y_i)}{f_1(Y_i)}\right)$, respectively. Also, the calculated variances for both are the same, i.e., $\sigma^2_x = \sigma^2_y = (\theta_0 - \theta_1)^2$.

In this context, from Theorem \ref{P2:Theorem 3.2}, the closed-form approximation $N_1^\ast$ is,
\begin{align*}
    \lim_{n \to \infty}\mathbb{E}(N_{1,n})
    & \approx \sum_{m=1}^{\infty} \Phi\left(- \frac{\eta_x}{\sigma_x}  \sqrt{m}\right) + \sum_{m=1}^{\infty} \Phi\left(\frac{\eta_y}{\sigma_y}  \sqrt{m}\right) \\\notag
    & \approx \frac{1}{2} \left(\frac{\sigma_x^2}{\eta_x^2} + \frac{\sigma_y^2}{\eta_y^2}\right) \\\notag
    & = N_1^\ast = \frac{4}{(\theta_0 - \theta_1)^2} < \infty.
\end{align*}

\begin{Remark} \label{P2:Remark 10}
    With $X_i \sim \operatorname{N}(\theta_0, 1)$ and $Y_i \sim \operatorname{N}(-\theta_0, 1)$, we have $\eta_x = 2\theta_0^2$ and $\eta_y = - 2\theta_0^2$, whereas, $\sigma^2_x = \sigma^2_y = 4\theta_0^2$ for doing adaptive sequential testing as discussed earlier. Here, we can note that, the allocation rule in this context does not depend on parameters and it depends only on $\bar{U_n} = \frac{1}{n}\sum_{i=1}^nU_i$ and $n$. In that case, $N_1^\ast = \frac{1}{\theta_0^2} < \infty$.
\end{Remark}

\begin{Remark} \label{P2:Remark 11}
    Let $f_0$, $f_1$ be in the same $\mathrm{MLR}$-family (with parameter $\theta$) with $\bar{X_n}$ and $\bar{Y_n}$ as sufficient statistics respectively. Note that, for testing composite hypotheses $H^\prime_0: \theta \leq \theta_0$ vs $H^\prime_1: \theta > \theta_1$, if we try to apply adaptive sequential rule as discussed earlier, allocation rule will depend on $\bar{U_n}$ and parameters, and $a$, $b$, $\alpha$, $\beta$ depend on parameters. Then, $N_1^\ast = \frac{1}{2} \left (\frac{\sigma_x^2 (\theta^\prime_0, \theta^\prime_1)}{\eta_x^2(\theta^\prime_0, \theta^\prime_1)} + \frac{\sigma_y^2 (\theta^\prime_0, \theta^\prime_1)}{\eta_y^2(\theta^\prime_0, \theta^\prime_1)}\right)$, where $\theta^\prime_0 \in H^\prime_0$ and $\theta^\prime_1 \in H^\prime_1$. In that case, for $\theta^\prime_0 \in H^\prime_0$ and $\theta^\prime_1 \in H^\prime_1$, the highest value of $N_1^\ast$ is approximately $\frac{1}{2} \left (\frac{\sigma_x^2 (\theta_0, \theta_1)}{\eta_x^2(\theta_0, \theta_1)} + \frac{\sigma_y^2 (\theta_0, \theta_1)}{\eta_y^2(\theta_0, \theta_1)}\right)$.
\end{Remark}

\section{Simulation Studies} \label{P2:Section 4}


In this section, we present a comprehensive simulation analysis to assess the performance of the proposed adaptive $\mathrm{SPRT}$ procedure under a range of underlying distributional settings. For each configuration, we estimate the $\mathrm{PCS}$, the expected number of allocations to the inferior population, and the $\mathrm{ASN}$. Simulations are conducted under Normal, Poisson, and Asymmetric Laplace distributions, with $1000$ replications in each scenario. These results collectively convey the operational behaviour of the procedure across distinct distributional regimes.

Across all experiments, the decision thresholds are computed as $a = \log \left(\frac{1-\beta}{\alpha} \right), \hspace{1mm} b = \log \left(\frac{\beta}{1-\alpha} \right)$, for each specified pair $(\alpha,\beta)$. Sampling commences when one observation is drawn from the population chosen at random, forming the initial likelihood contributions. Subsequent sampling proceeds according to the adaptive allocation rule described in Section \ref{P2:Section 2.2}. At each stage, the accumulated sample sizes from the two populations are examined, and the next observation is drawn from the population whose cumulative log-likelihood ratio $L_n = \sum_{i=1}^{n_{\max}} \log\!\left(\frac{f_0(U_i)}{f_1(U_i)}\right)$, provides less support for $K_1$ compared to $K_0$ (hypotheses as defined in Section \ref{P2:Section 3.3}). This mechanism ensures that sampling is dynamically steered toward the superior population, thereby allowing the likelihood ratio to evolve in an efficient manner towards getting more and more samples from the superior population and highlighting the implication of the adaptive nature of the design.

After each new observation, the log-likelihood ratio is updated based on the stream currently yielding the larger sample size. The procedure terminates once the statistic crosses one of the two boundaries: the alternative $K_1$ is accepted when the statistic exceeds $a$, and the null $K_0$ is accepted when it falls below $b$. For every replication, we record whether the final decision corresponds to the truly superior population, the total number of observations drawn, and the frequency of allocations to the inferior population. Averages across replications yield the performance metrics $\mathrm{PCS}$, $\mathbb{E}(N_{1,n})$, and $\mathrm{ASN}$. Moreover, we denote the closed-form approximation of the limiting value of $E(N_{1,n})$ by $N_1^*$ and report it for comparison with the value obtained in the corresponding simulation results.

\begin{Remark} \label{P2:Remark 12}
    Note that, in each of the simulation tables, we have listed $\alpha\,(=\beta)$ in the first column only to indicate the error probabilities that would be targeted in the usual $\mathrm{SPRT}$, with the corresponding choices of the boundary constants $a$ and $b$. Effectively, the probabilities of correct selection obtained in the respective tables do not match the values of $\alpha\,(=\beta)$ given in the first column. This is an expected result, as the effective average sample number in the adaptive $\mathrm{SPRT}$ is compounded with $N_1^*$, which contributes to a much higher level of the error probabilities, owing to the exponential effect of the average sample number. Nevertheless, this is consistent with our result that the adaptive $\mathrm{SPRT}$ is asymptotically as efficient as the usual $\mathrm{SPRT}$.
\end{Remark}

The subsequent subsections outline the distributional scenarios considered and summarize the corresponding numerical results.

\subsection{Normal Distributions (Adaptive \texorpdfstring{$\mathrm{SPRT}$}{SPRT})} \label{P2:Section 4.1}

We first consider $f_0 \sim \operatorname{N}(\theta_0,1), \hspace{1mm} f_1 \sim \operatorname{N}(\theta_1,1)$, with the following mean pairs
$$
(\theta_0,\theta_1) \in \{(0.1,0), (0.2,0), (0.3,0), (0.4,0), (0.5,0)\}.
$$

Results of this simulation study are reported in \autoref{P2:Table 1}. The adaptive $\mathrm{SPRT}$ procedure continues to perform reliably, producing high $\mathrm{PCS}$ and maintaining small inferior allocations across these configurations. The expected number of allocations to the inferior population remains small and aligns closely with the theoretical benchmark $N_1^\ast$ derived in Section \ref{P2:Section 3.5}. As the separation $|\theta_0 - \theta_1|$ increases, $\mathrm{ASN}$ declines substantially, and as $(\alpha, \beta)$ decreases, $\mathbb{E}(N_{1,n})$ approaches $N_1^\ast$, reflecting good discrimination between the two populations.

\begin{table*}[ht]
\centering
\caption{Simulation results for two Normal populations with distinct mean pairs $(\theta_{0}, \theta_{1})$ and common variance $(\sigma_{0}^2, \sigma_{1}^2) = (1, 1)$ conducted within the adaptive $\mathrm{SPRT}$ framework.}

\label{P2:Table 1}

\subfloat[$(\theta_0, \theta_1) = (0.1, 0)$, $N_1^\ast = 400$]{
\resizebox{0.45\textwidth}{!}{%
\begin{tabular}{|c|ccc|}
\hline
$\alpha(=\beta)$ & $\mathrm{PCS}$ & $\mathbb{E}(N_{1,n})$ & $\mathrm{ASN}$ \\ \hline
$10^{-3}$         & 0.908 & 349.550 & 1575.253 \\
$5\times10^{-5}$  & 0.955 & 394.827 & 2289.517 \\
$10^{-5}$         & 0.969 & 391.828 & 2657.338 \\
$5\times10^{-6}$  & 0.975 & 398.996 & 2756.844 \\
$10^{-6}$         & 0.986 & 394.102 & 3093.543 \\ \hline
\end{tabular}%
}}
\hspace{5mm}
\subfloat[$(\theta_0, \theta_1) = (0.2, 0)$, $N_1^\ast = 100$]{
\resizebox{0.45\textwidth}{!}{%
\begin{tabular}{|c|ccc|}
\hline
$\alpha(=\beta)$ & $\mathrm{PCS}$ & $\mathbb{E}(N_{1,n})$ & $\mathrm{ASN}$ \\ \hline
$10^{-3}$         & 0.918 & 91.237 & 418.085 \\
$5\times10^{-5}$  & 0.955 & 98.748  & 574.303 \\
$10^{-5}$         & 0.979 & 97.999  & 663.545 \\
$5\times10^{-6}$  & 0.976 & 99.725  & 694.320 \\
$10^{-6}$         & 0.984 & 100.309  & 781.425 \\ \hline
\end{tabular}%
}}

\vspace{4mm}

\subfloat[$(\theta_0, \theta_1) = (0.3, 0)$, $N_1^\ast = 44.444$]{
\resizebox{0.45\textwidth}{!}{%
\begin{tabular}{|c|ccc|}
\hline
$\alpha(=\beta)$ & $\mathrm{PCS}$ & $\mathbb{E}(N_{1,n})$ & $\mathrm{ASN}$ \\ \hline
$10^{-3}$         & 0.909 & 38.350 & 180.370 \\
$5\times10^{-5}$  & 0.969 & 38.488 & 252.804 \\
$10^{-5}$         & 0.974 & 46.189 & 296.015 \\
$5\times10^{-6}$  & 0.975 & 43.336 & 309.886 \\
$10^{-6}$         & 0.989 & 40.482 & 351.835 \\ \hline
\end{tabular}%
}}
\hspace{5mm}
\subfloat[$(\theta_0, \theta_1) = (0.4, 0)$, $N_1^\ast = 25$]{
\resizebox{0.45\textwidth}{!}{%
\begin{tabular}{|c|ccc|}
\hline
$\alpha(=\beta)$ & $\mathrm{PCS}$ & $\mathbb{E}(N_{1,n})$ & $\mathrm{ASN}$ \\ \hline
$10^{-3}$         & 0.919 & 20.496 & 100.263 \\
$5\times10^{-5}$  & 0.952 & 23.338 & 143.625 \\
$10^{-5}$         & 0.971 & 25.552 & 168.346 \\
$5\times10^{-6}$  & 0.987 & 23.845 & 177.570 \\
$10^{-6}$         & 0.984 & 22.652 & 194.418 \\
 \hline
\end{tabular}%
}}

\vspace{4mm}

\subfloat[$(\theta_0, \theta_1) = (0.5, 0)$, $N_1^\ast = 16$]{
\resizebox{0.45\textwidth}{!}{%
\begin{tabular}{|c|ccc|}
\hline
$\alpha(=\beta)$ & $\mathrm{PCS}$ & $\mathbb{E}(N_{1,n})$ & $\mathrm{ASN}$ \\ \hline
$10^{-3}$         & 0.930 & 13.422 & 66.488 \\
$5\times10^{-5}$  & 0.961 & 15.735 & 94.617 \\
$10^{-5}$         & 0.980 & 14.324 & 105.348 \\
$5\times10^{-6}$  & 0.979 & 14.463 & 110.278 \\
$10^{-6}$         & 0.985 & 15.161 & 124.351 \\
 \hline
\end{tabular}%
}}
\end{table*}

\subsection{Poisson Distributions (Adaptive \texorpdfstring{$\mathrm{SPRT}$}{SPRT})} \label{P2:Section 4.2}

Next, we consider $f_0 \sim \operatorname{P}(\lambda_0), \hspace{1mm}
f_1 \sim \operatorname{P}(\lambda_1)$, with several contrasting parameter pairs
$$
(\lambda_0, \lambda_1) \in \{(2.5, 2), (3, 2.5), (3.5, 2.5), (2, 1), (1.5, 0.5), (2.5, 1)\}.
$$

The adaptive procedure continues to exhibit good behaviour under these discrete distributions, yielding high $\mathrm{PCS}$ values and maintaining modest inferior allocations. As in the Normal case, $\mathrm{ASN}$ decreases as the divergence between $\lambda_0$ and $\lambda_1$ widens. The corresponding numerical outcomes are displayed in \autoref{P2:Table 2}. 

\begin{table*}[h]
\centering
\caption{Simulation results for two Poisson populations with distinct mean pairs $(\lambda_{0}, \lambda_{1})$ conducted within the adaptive $\mathrm{SPRT}$ framework.}
\label{P2:Table 2}

\subfloat[$(\lambda_0, \lambda_1) = (2.5, 2)$, $N_1^\ast = 35.851$]{
\resizebox{0.45\textwidth}{!}{%
\begin{tabular}{|c|ccc|}
\hline
$\alpha(=\beta)$ & $\mathrm{PCS}$ & $\mathbb{E}(N_{1,n})$ & $\mathrm{ASN}$ \\ \hline
$10^{-3}$         & 0.924 & 32.295 & 146.120 \\
$5\times10^{-5}$  & 0.959 & 35.314 & 203.890 \\
$10^{-5}$         & 0.984 & 32.991 & 229.243 \\
$5\times10^{-6}$  & 0.982 & 35.608 & 245.363 \\
$10^{-6}$         & 0.987 & 34.219 & 271.723 \\ \hline
\end{tabular}%
}}
\hspace{5mm}
\subfloat[$(\lambda_0, \lambda_1) = (3, 2.5)$, $N_1^\ast = 43.879$]{
\resizebox{0.45\textwidth}{!}{%
\begin{tabular}{|c|ccc|}
\hline
$\alpha(=\beta)$ & $\mathrm{PCS}$ & $\mathbb{E}(N_{1,n})$ & $\mathrm{ASN}$ \\ \hline
$10^{-3}$         & 0.909 & 39.151 & 176.486 \\
$5\times10^{-5}$  & 0.971 & 41.286 & 246.989 \\
$10^{-5}$         & 0.980 & 42.927 & 286.710 \\
$5\times10^{-6}$  & 0.978 & 43.234 & 303.215 \\
$10^{-6}$         & 0.986 & 42.506 & 333.556 \\ \hline
\end{tabular}%
}}

\vspace{4mm}

\subfloat[$(\lambda_0, \lambda_1) = (3.5, 2.5)$, $N_1^\ast = 11.888$]{
\resizebox{0.45\textwidth}{!}{%
\begin{tabular}{|c|ccc|}
\hline
$\alpha(=\beta)$ & $\mathrm{PCS}$ & $\mathbb{E}(N_{1,n})$ & $\mathrm{ASN}$ \\ \hline
$10^{-3}$         & 0.931 & 10.940 & 49.169 \\
$5\times10^{-5}$  & 0.967 & 10.668 & 66.835 \\
$10^{-5}$         & 0.980 & 11.002 & 76.338 \\
$5\times10^{-6}$  & 0.987 & 10.611 & 80.323 \\
$10^{-6}$         & 0.996 & 10.525 & 89.907 \\
 \hline
\end{tabular}%
}}
\hspace{5mm}
\subfloat[$(\lambda_0, \lambda_1) = (2, 1)$, $N_1^\ast = 5.771$]{
\resizebox{0.45\textwidth}{!}{%
\begin{tabular}{|c|ccc|}
\hline
$\alpha(=\beta)$ & $\mathrm{PCS}$ & $\mathbb{E}(N_{1,n})$ & $\mathrm{ASN}$ \\ \hline
$10^{-3}$         & 0.935 & 5.452 & 24.310 \\
$5\times10^{-5}$  & 0.969 & 5.548 & 32.654 \\
$10^{-5}$         & 0.981 & 5.793 & 37.123 \\
$5\times10^{-6}$  & 0.986 & 5.962 & 39.403 \\
$10^{-6}$         & 0.992 & 5.518 & 42.659 \\
 \hline
\end{tabular}%
}}

\vspace{4mm}

\subfloat[$(\lambda_0, \lambda_1) = (1.5, 0.5)$, $N_1^\ast = 3.642$]{
\resizebox{0.45\textwidth}{!}{%
\begin{tabular}{|c|ccc|}
\hline
$\alpha(=\beta)$ & $\mathrm{PCS}$ & $\mathbb{E}(N_{1,n})$ & $\mathrm{ASN}$ \\ \hline
$10^{-3}$         & 0.962 & 3.662 & 15.269 \\
$5\times10^{-5}$  & 0.981 & 3.759 & 20.489 \\
$10^{-5}$         & 0.989 & 3.689 & 22.976 \\
$5\times10^{-6}$  & 0.986 & 3.650 & 24.098 \\
$10^{-6}$         & 0.994 & 3.701 & 26.718 \\
 \hline
\end{tabular}%
}}
\hspace{5mm}
\subfloat[$(\lambda_0, \lambda_1) = (2.5, 1)$, $N_1^\ast = 2.911$]{
\resizebox{0.45\textwidth}{!}{%
\begin{tabular}{|c|ccc|}
\hline
$\alpha(=\beta)$ & $\mathrm{PCS}$ & $\mathbb{E}(N_{1,n})$ & $\mathrm{ASN}$ \\ \hline
$10^{-3}$         & 0.952 & 2.996 & 12.614 \\
$5\times10^{-5}$  & 0.977 & 3.093 & 16.832 \\
$10^{-5}$         & 0.986 & 3.204 & 19.170 \\
$5\times10^{-6}$  & 0.990 & 2.991 & 19.950 \\
$10^{-6}$         & 0.989 & 3.104 & 21.790 \\
 \hline
\end{tabular}%
}}
\end{table*}

\subsection{Asymmetric Laplace Distributions (Adaptive \texorpdfstring{$\mathrm{SPRT}$}{SPRT})} \label{P2:Section 4.3}

To evaluate performance under skewed and asymmetric settings, we consider the Asymmetric Laplace family \citep{AsymmetricLaplace2001}
\[
f(x;m,\lambda,\kappa)
= \frac{\lambda}{\kappa + \kappa^{-1}}
\begin{cases}
\exp\!\left(\dfrac{\lambda}{\kappa}(x - m)\right), & x < m, \\[6pt]
\exp\!\left(-\lambda\kappa(x - m)\right), & x \ge m,
\end{cases}
\]
where $m$ is a location parameter, $\lambda > 0$ is a scale parameter, and $\kappa > 0$ governs the degree of asymmetry.

A variety of contrasting parameter pairs $(m_0,\lambda_0,\kappa_0)$ and $(m_1,\lambda_1,\kappa_1)$ are examined. Across all configurations, the adaptive $\mathrm{SPRT}$ continues to achieve strong $\mathrm{PCS}$ performance and small inferior allocations. As expected, $\mathrm{ASN}$ decreases monotonically with increasing separation between the distributions. The resulting performance measures are summarized in \autoref{P2:Table 3}.

\begin{table*}[h]
\centering
\caption{Simulation results for two Asymmetric Laplace populations with various parameter configurations $(m_0, \lambda_0, \kappa_0)$ and $(m_1, \lambda_1, \kappa_1)$ conducted within the adaptive $\mathrm{SPRT}$ framework.}
\label{P2:Table 3}

\subfloat[$(m_0, \lambda_0, \kappa_0) = (0.2, 2, 0.7)$, $(m_1, \lambda_1, \kappa_1) = (0, 1, 0.3)$, $N_1^\ast = 2.288$]{
\resizebox{0.45\textwidth}{!}{%
\begin{tabular}{|c|ccc|}
\hline
$\alpha(=\beta)$ & $\mathrm{PCS}$ & $\mathbb{E}(N_{1,n})$ & $\mathrm{ASN}$ \\ \hline
$10^{-3}$         & 0.844 & 1.936 & 11.701 \\
$10^{-5}$         & 0.920 & 1.996 & 19.366 \\
$5\times10^{-6}$  & 0.949 & 1.978 & 20.703 \\
$10^{-6}$         & 0.955 & 2.020 & 23.325 \\
$10^{-7}$         & 0.964 & 2.060 & 27.126 \\ \hline
\end{tabular}%
}}
\hspace{5mm}
\subfloat[$(m_0, \lambda_0, \kappa_0) = (0.2, 1, 0.8)$, $(m_1, \lambda_1, \kappa_1) = (0, 2, 0.2)$, $N_1^\ast = 4.802$]{
\resizebox{0.45\textwidth}{!}{%
\begin{tabular}{|c|ccc|}
\hline
$\alpha(=\beta)$ & $\mathrm{PCS}$ & $\mathbb{E}(N_{1,n})$ & $\mathrm{ASN}$ \\ \hline
$10^{-3}$         & 0.943 & 3.259 & 9.819 \\
$10^{-5}$         & 0.985 & 3.493 & 12.693 \\
$5\times10^{-6}$  & 0.989 & 3.416 & 12.720 \\
$10^{-6}$         & 0.994 & 3.364 & 13.489 \\
$10^{-7}$         & 0.995 & 3.511 & 14.804 \\ \hline
\end{tabular}%
}}

\vspace{4mm}

\subfloat[$(m_0, \lambda_0, \kappa_0) = (0.4, 1, 0.6)$, $(m_1, \lambda_1, \kappa_1) = (0, 1, 0.2)$, $N_1^\ast = 4.576$]{
\resizebox{0.45\textwidth}{!}{%
\begin{tabular}{|c|ccc|}
\hline
$\alpha(=\beta)$ & $\mathrm{PCS}$ & $\mathbb{E}(N_{1,n})$ & $\mathrm{ASN}$ \\ \hline
$10^{-3}$         & 0.893 & 2.793 & 15.669 \\
$10^{-5}$         & 0.959 & 2.890 & 24.958 \\
$5\times10^{-6}$  & 0.969 & 2.812 & 26.276 \\
$10^{-6}$         & 0.971 & 3.005 & 29.531 \\
$10^{-7}$         & 0.975 & 3.046 & 33.611 \\ \hline
\end{tabular}%
}}
\hspace{5mm}
\subfloat[$(m_0, \lambda_0, \kappa_0) = (0, 2, 0.7)$, $(m_1, \lambda_1, \kappa_1) = (0.2, 2, 0.3)$, $N_1^\ast = 2.774$]{
\resizebox{0.45\textwidth}{!}{%
\begin{tabular}{|c|ccc|}
\hline
$\alpha(=\beta)$ & $\mathrm{PCS}$ & $\mathbb{E}(N_{1,n})$ & $\mathrm{ASN}$ \\ \hline
$10^{-3}$         & 0.940 & 2.323 & 9.663 \\
$10^{-5}$         & 0.975 & 2.492 & 14.553 \\
$5\times10^{-6}$  & 0.986 & 2.539 & 15.371 \\
$10^{-6}$         & 0.989 & 2.571 & 17.121 \\
$10^{-7}$         & 0.992 & 2.514 & 18.929 \\ \hline
\end{tabular}%
}}
\end{table*}

\begin{Remark} \label{P2:Remark 13}
    $N_1^\ast = \frac{1}{2}\!\left(\frac{\sigma_x^{2}}{\eta_x^{2}} + \frac{\sigma_y^{2}}{\eta_y^{2}}\right)$ approximates the limit of $\mathbb{E}(N_{1,n})$, i.e., an approximation to the finite limit of applications to the inferior treatment. For adaptive $\mathrm{SPRT}$, in Tables \ref{P2:Table 1}, \ref{P2:Table 2} and \ref{P2:Table 3}, we have found that $\mathbb{E}(N_{1,n})$ conforms to $N_1^\ast$ as $(\alpha, \beta)$ decreases. There may be small fluctuations due to sampling error but overall $\mathbb{E}(N_{1,n})$ goes close to $N_1^\ast$.
\end{Remark}

\subsection{Comparison with classical \texorpdfstring{$\mathrm{SPRT}$}{SPRT}} \label{P2:Section 4.4}

For simplicity, we also examine the classical $\mathrm{SPRT}$ under the Normal distribution.

Let $\{X_i\}$ be the data stream for the two competing hypotheses $H_0: \mu = \mu_0$ vs $H_1: \mu = \mu_1$, where $X_i \sim \operatorname{N}(\mu, 1)$. This test utilizes the same threshold pair $(a, b)$, i.e., the same pair of error probabilities $(\alpha, \beta)$. The test statistic is $S_n = \sum_{i=1}^{n} \log\!\left( \frac{f(X_i - \mu_1)}{f(X_i - \mu_0)} \right)$, where $f$ is the standard normal density. The sampling continues until $S_n$ crosses one of the stopping boundaries.

\begin{Remark} \label{P2:Remark 14}
    The test in the context of the adaptive procedure effectively compares the two hypotheses $H_0: (\theta_0, \theta_1) = (\theta_0^*, \theta_1^*)$ vs $H_1: (\theta_0, \theta_1) = (\theta_1^*, \theta_0^*)$, where $\theta_0$ and $\theta_1$ are the parameters of the two respective densities, possibly from the same family. In fact, a two-population classical $\mathrm{SPRT}$ may also be constructed for the above testing problem. The underlying frameworks of the $\mathrm{SPRT}$ in the single-parameter and two-parameter contexts are effectively very much comparable in their reasoning as inference problems.

    In \autoref{P2:Table 4}, we present the classical $\mathrm{SPRT}$ in a one-parameter family for the above reason and also for simplicity.
\end{Remark}

Although the classical $\mathrm{SPRT}$ often achieves reasonably small $\mathrm{ASN}$, it necessarily allocates a substantial number of samples to the inferior population. This stands in sharp contrast with the adaptive $\mathrm{SPRT}$, which significantly curtails inferior allocations by design. The numerical results (\autoref{P2:Table 4}) show that the $\mathrm{PCS}$ of classical $\mathrm{SPRT}$ is higher than that of our adaptive $\mathrm{SPRT}$ method, although both methods are asymptotically efficient. 

\begin{table*}[ht]
\centering
\caption{Simulation results for two Normal populations with distinct mean pairs $(\mu_{0}, \mu_{1})$ and common variance $(\sigma_{0}^2, \sigma_{1}^2) = (1, 1)$ conducted within the classical $\mathrm{SPRT}$ framework.}
\label{P2:Table 4}
\resizebox{\textwidth}{!}{%
\begin{tabular}{| l | cc | cc | cc | cc | cc |}
\hline
& \multicolumn{2}{c|}{$(\mu_0, \mu_1) = (0.1, 0)$} 
& \multicolumn{2}{c|}{$(\mu_0, \mu_1) = (0.2, 0)$}
& \multicolumn{2}{c|}{$(\mu_0, \mu_1) = (0.3, 0)$}
& \multicolumn{2}{c|}{$(\mu_0, \mu_1) = (0.4, 0)$}
& \multicolumn{2}{c|}{$(\mu_0, \mu_1) = (0.5, 0)$} \\
\cline{2-11}
$\alpha(=\beta)$ & $\mathrm{PCS}$ & $\mathrm{ASN}$ 
           & $\mathrm{PCS}$ & $\mathrm{ASN}$ 
           & $\mathrm{PCS}$ & $\mathrm{ASN}$ 
           & $\mathrm{PCS}$ & $\mathrm{ASN}$ 
           & $\mathrm{PCS}$ & $\mathrm{ASN}$ \\ 
\hline
$10^{-2}$ & 0.989 & 928.385 & 0.991 & 231.763 & 0.990 & 104.929 & 0.992 & 57.987 & 0.989 & 38.152 \\
$10^{-3}$ & 0.999 & 1370.521 & 1.000 & 346.063 & 1.000 & 155.557 & 0.999 & 89.838 & 0.998 & 56.901 \\
$10^{-4}$ & 0.999 & 1867.227 & 1.000 & 462.753 & 1.000 & 206.995 & 0.999 & 117.981 & 1.000 & 76.770 \\
$10^{-5}$ & 1.000 & 2335.468 & 1.000 & 571.699 & 1.000 & 265.242 & 1.000 & 146.960 & 1.000 & 95.605 \\
\hline
\end{tabular}%
}
\end{table*}

\begin{Remark} \label{P2:Remark 15}
    Regarding $\mathrm{PCS}$ of adaptive $\mathrm{SPRT}$ in Tables \ref{P2:Table 1}, \ref{P2:Table 2} and \ref{P2:Table 3}, we have seen that the obtained $\mathrm{PCS}$ is a little lower than $1-\alpha (= 1-\beta)$, though it is found that $\mathrm{PCS}$ tends to $1$ as $\mathrm{ASN}$ increases (confirming that adaptive $\mathrm{SPRT}$ is asymptotically efficient). In classical $\mathrm{SPRT}$ (\autoref{P2:Table 4}), the value of $\mathrm{PCS}$ coincides with $1-\alpha (= 1-\beta)$. Effective $\mathrm{ASN}$ of adaptive $\mathrm{SPRT}$ is less than that of classical $\mathrm{SPRT}$ by the finite amount $N_1^\ast$. Hence, though being asymptotically efficient, adaptive $\mathrm{SPRT}$ gives a slightly higher value of probability of incorrect selection.
\end{Remark}

\subsection{Summary} \label{P2:Section 4.5}

Across all distributional regimes, the adaptive $\mathrm{SPRT}$ maintains high $\mathrm{PCS}$ while drastically reducing sampling from the inferior population. $\mathrm{ASN}$ decreases systematically with increasing signal strength, and the procedure remains stable across symmetric, discrete, and skewed scenarios. In the Normal case, the classical $\mathrm{SPRT}$ underscores the ethical and operational advantages of the adaptive allocation mechanism while retaining comparable inferential accuracy, the adaptive $\mathrm{SPRT}$ dramatically reduces the expected inferior sample size. These findings collectively highlight the strong practical merits of the adaptive $\mathrm{SPRT}$ framework in sequential decision-making problems.

\section{Data-Calibrated Case Studies} \label{P2:Section 5}


To illustrate the proposed adaptive $\mathrm{SPRT}$ in realistic clinical settings, we consider two published clinical trials. Since these trials were conducted with non-adaptive allocation, the responses that the patients would have had under the other treatment are not observed; hence, the original data cannot be re-analysed directly under an adaptive allocation rule. Instead, we use the published estimates as the true parameters of $f_0$ and $f_1$ and redesign each trial by simulation. These analyses are therefore data-calibrated case studies.

\subsection{Pregabalin Drug Trial for Postherpetic Neuralgia (\texorpdfstring{$\mathrm{PHN}$}{PHN})} \label{P2:Section 5.1}

To exemplify the practical utility of the proposed adaptive $\mathrm{SPRT}$, we consider the randomized, double-blind, placebo-controlled clinical trial to appraise the potency of the drug Pregabalin for the treatment of postherpetic neuralgia ($\mathrm{PHN}$), conducted by \citet{Dworkinetal2003}. $\mathrm{PHN}$ is a chronic neuropathic pain condition that persists following herpes zoster infection, and the primary efficacy endpoint of the study was the mean pain score at the end of the treatment period. A total of $173$ patients were randomized to receive either Pregabalin or placebo. The reported endpoint mean pain scores were $3.60$ for the Pregabalin group and $5.29$ for the placebo group, demonstrating significantly greater pain reduction in the Pregabalin arm. Owing to its importance in illustrating response-adaptive allocation procedures, this clinical trial has also been discussed by \citet{BiswasBandyopadhyayBhattacharya2008} as a representative example of adaptive Phase III clinical trial designs.

As the individual patient-level data are unavailable, we use the published summary statistics to characterize the underlying response distributions. To adapt this continuous-response setting to the proposed adaptive $\mathrm{SPRT}$ framework developed in Sections \ref{P2:Section 2}--\ref{P2:Section 3}, the responses from the Pregabalin and placebo groups are modeled as independent Normal distributions, with means and standard deviations taken from the reported endpoint pain scores. Since lower pain scores indicate superior treatment efficacy, whereas the proposed adaptive sequential framework is developed under the convention that larger location parameters correspond to more effective treatments, we transform the response variable considering the negative of the observed pain scores. Specifically, we use the following values as the true parameters:
\begin{itemize}
    \item Pregabalin (treatment $0$): mean $\theta_0 = -3.60$, sd $\sigma_0 = 2.25$,
    \item Placebo (treatment $1$): mean $\theta_1 = -5.29$, sd $\sigma_1 = 2.20$.
\end{itemize} These parameter estimates are regarded as the face values of the true parameters, and the original Pregabalin trial is redesigned under the proposed adaptive $\mathrm{SPRT}$ framework. The redesigned trial is then independently simulated for $1000$ replications. Accordingly, the corresponding empirical values of $\mathrm{PCS}$, $\mathbb{E}(N_{1,n})$ and $\mathrm{ASN}$ are computed and recorded.

\autoref{P2:Table 5} summarizes the empirical performance of the redesigned trial under different choices of the error probabilities. The empirical results indicate that the proposed adaptive $\mathrm{SPRT}$ consistently achieves a high probability of correct selection while maintaining a finite expected number of allocations to the inferior treatment. Moreover, the empirical values of $\mathbb{E}(N_{1,n})$ remain close to, but below, the closed-form approximation $N_1^\ast=11.504$, thereby corroborating the theoretical findings of Section \ref{P2:Section 3}. These results demonstrate that the proposed adaptive $\mathrm{SPRT}$ effectively balances ethical considerations and inferential efficiency by substantially reducing patient exposure to the clinically inferior placebo treatment without sacrificing sequential testing performance.

\begin{table}[!htbp]
\centering
\caption{Results for Pregabalin trial dataset. The closed-form approximation} is $N_1^\ast=11.504$.
\label{P2:Table 5}

\begin{tabular}{|c|ccc|}
\hline
$\alpha(=\beta)$ & $\mathrm{PCS}$ & $\mathbb{E}(N_{1,n})$ & $\mathrm{ASN}$\\
\hline
$10^{-3}$        & 0.942 & 6.067 & 28.815\\
$5\times10^{-5}$ & 0.972 & 6.863 & 41.859\\
$10^{-5}$        & 0.984 & 6.453 & 45.843\\
$5\times10^{-6}$ & 0.977 & 6.871 & 48.785\\
$10^{-6}$        & 0.988 & 6.699 & 54.628\\
\hline
\end{tabular}

\end{table}

\subsection{Epileptic Seizure Data} \label{P2:Section 5.2}

To further illustrate the practical utility of the proposed adaptive $\mathrm{SPRT}$ for count data, we consider the epileptic seizure dataset arising from the randomized, double-blind, placebo-controlled clinical trial conducted by \citet{Leppiketal1985}, in which Progabide was investigated as an adjunctive therapy and compared with placebo, for patients with partial epilepsy. This dataset has subsequently been reported by \citet{ThallVali1990} and has become a standard benchmark for illustrating statistical methodologies involving count responses.

The dataset consists of $59$ patients randomly assigned to either the Progabide or placebo group. For each patient, seizure counts were recorded at four successive post-treatment follow-up visits. Since a lower seizure count indicates superior treatment efficacy, the treatment with the smaller mean seizure count is regarded as the superior treatment. Although the original dataset is longitudinal, the objective of the present analysis is not to model the dependence among repeated observations but rather to illustrate the implementation of the proposed adaptive $\mathrm{SPRT}$ for count data. Possible overdispersion is likewise not modelled; hence, this analysis serves as an illustration of the procedure. Accordingly, the seizure counts across the four follow-up visits are summed to obtain the total post-treatment seizure count for each patient. The resulting patient-specific total seizure counts in the Progabide and placebo groups are modeled by Poisson distributions, with the corresponding Poisson means estimated by the sample mean total seizure counts in the respective groups. These parameter estimates are regarded as the true parameter values, and the original clinical trial is redesigned under the proposed adaptive $\mathrm{SPRT}$ framework. The operating characteristics of the redesigned trial are then investigated, with particular emphasis on reducing allocations to the inferior treatment while maintaining reliable inferential performance.

The treatment-specific Poisson parameters are considered as the underlying true parameter values for the redesigned experiment. The estimated mean seizure counts are:
\begin{itemize}
    \item Progabide (treatment $0$): $\lambda_0 = 31.8387$,
    \item Placebo (treatment $1$): $\lambda_1 = 34.3929$.
\end{itemize}

Although $\lambda_0 < \lambda_1$, treatment $0$ is designated as the superior treatment because lower seizure counts indicate better clinical efficacy.

Using these estimated parameter values, we redesign the trial and implement the proposed adaptive $\mathrm{SPRT}$ for several choices of the error probabilities $\alpha=\beta$. For each configuration, $1000$ independent replications are performed. In this context, the empirical values of $\mathrm{PCS}$, $E(N_{1,n})$ and $\mathrm{ASN}$ are computed and reported.

The results are presented in \autoref{P2:Table 6}. It is observed that the empirical values of $E(N_{1,n})$ remain close to the closed-form approximation $N_1^{*} \hspace{2mm} (20.294)$, thereby providing strong empirical support for the theoretical approximation derived in Theorem \ref{P2:Theorem 3.2}. Furthermore, the $\mathrm{PCS}$ increases steadily as $\alpha$ decreases, indicating improved reliability of the adaptive decision rule. The $\mathrm{ASN}$ increases owing to the more stringent stopping boundaries. Nevertheless, the expected number of allocations to the clinically inferior placebo treatment fluctuates around the closed-form approximation, demonstrating that the proposed adaptive $\mathrm{SPRT}$ established in Section \ref{P2:Section 3} effectively reduces patient exposure to the less effective treatment while maintaining reliable inferential performance in a realistic clinical trial setting.

\begin{table}[!htbp]
\centering
\caption{Results for the epilepsy (Progabide--Placebo) trial dataset. The closed-form approximation is $N_1^\ast=20.294$.}
\label{P2:Table 6}
\begin{tabular}{|c|ccc|}
\hline
$\alpha(=\beta)$ & $\mathrm{PCS}$ & $\mathbb{E}(N_{1,n})$ & $\mathrm{ASN}$\\
\hline
$10^{-3}$        & 0.911 & 18.906 & 87.526\\
$5\times10^{-5}$ & 0.954 & 19.240 & 119.344\\
$10^{-5}$        & 0.975 & 19.978 & 136.637\\
$5\times10^{-6}$ & 0.977 & 19.841 & 144.216\\
$10^{-6}$        & 0.984 & 20.035 & 162.009\\
\hline
\end{tabular}
\end{table}

\section{Concluding Remarks} \label{P2:Section 6}

This work, from a broader perspective, addresses a longstanding gap in adaptive sequential analysis — the lack of an explicit, finite-form quantification of ethical performance in allocation-driven testing procedures. For testing $H_0:(f_0,f_1)$ against $H_1:(f_1,f_0)$, the proposed likelihood ratio--based allocation rule keeps the expected number of allocations to the less effective treatment finite: we establish an explicit upper bound for it and a closed-form  analytical approximation, which depends only on the means and variances of the log-likelihood ratio statistics. The procedure also has significant implications for experimental design in practice. In clinical trials or adaptive testing problems where sample collection incurs real-world ethical or economic costs, having an explicit upper bound on the expected number of allocations to the inferior option offers clear interpretability for regulators and practitioners. We have shown a very important application of the procedure by adjusting the classical $\mathrm{SPRT}$ based on the proposed adaptive rule. The results demonstrate that our procedure retains the asymptotic efficiency of the classical $\mathrm{SPRT}$, hence establishing its statistical control similar to the classical $\mathrm{SPRT}$, along with the ethical efficiency, which is lacking in the classical $\mathrm{SPRT}$. The data-calibrated case studies further enrich the study by providing empirical evidence that complements the theoretical developments and simulation findings in a realistic setting.

The study also opens up several directions for further research. In particular, when the parameters are completely unknown (composite hypotheses), it is of interest to design or modify the procedure to achieve a parameter-free decision rule while preserving the essential goals of ethical and inferential efficiency. In this paper, we have achieved this property to some extent (as mentioned in Remark \ref{P2:Remark 10}) for the symmetric case $H_0: \theta = \theta_0$ vs $H_1: \theta = - \theta_0$. It is to be seen if the results can also be extended to settings involving more than two treatments, where similar finiteness and asymptotic efficiency properties are expected to hold.

\bibliography{references}

\end{document}